\documentclass[fleqn,reqno,10pt,a4paper,final]{amsart}

\usepackage[a4paper,left=35mm,right=35mm,top=30mm,bottom=30mm,marginpar=25mm]{geometry} 
\usepackage{amsmath}
\usepackage{amssymb}
\usepackage{amsthm}
\usepackage{amscd}
\usepackage[ansinew]{inputenc}
\usepackage{cite}
\usepackage{bbm}
\usepackage{color}
\usepackage[english=american]{csquotes}
\usepackage[final]{graphicx}
\usepackage[colorlinks=true,linkcolor=red,citecolor=blue]{hyperref}

\graphicspath{{Pictures/}}

\numberwithin{equation}{section}

\newtheoremstyle{thmlemcorr}{10pt}{10pt}{\itshape}{}{\bfseries}{.}{10pt}{{\thmname{#1}\thmnumber{ #2}\thmnote{ (#3)}}}
\newtheoremstyle{thmlemcorr*}{10pt}{10pt}{\itshape}{}{\bfseries}{.}\newline{{\thmname{#1}\thmnumber{ #2}\thmnote{ (#3)}}}
\newtheoremstyle{defi}{10pt}{10pt}{\itshape}{}{\bfseries}{.}{10pt}{{\thmname{#1}\thmnumber{ #2}\thmnote{ (#3)}}}
\newtheoremstyle{remexample}{10pt}{10pt}{}{}{\bfseries}{.}{10pt}{{\thmname{#1}\thmnumber{ #2}\thmnote{ (#3)}}}
\newtheoremstyle{ass}{10pt}{10pt}{}{}{\bfseries}{.}{10pt}{{\thmname{#1}\thmnumber{ A#2}\thmnote{ (#3)}}}

\theoremstyle{thmlemcorr}
\newtheorem{theorem}{Theorem}
\numberwithin{theorem}{section}
\newtheorem{lemma}[theorem]{Lemma}

\newtheorem{proposition}[theorem]{Proposition}

\theoremstyle{thmlemcorr*}
\newtheorem{theorem*}{Theorem}
\newtheorem{lemma*}[theorem]{Lemma}
\newtheorem{corollary*}[theorem]{Corollary}
\newtheorem{proposition*}[theorem]{Proposition}
\newtheorem{problem*}[theorem]{Problem}
\newtheorem{conjecture*}[theorem]{Conjecture}

\theoremstyle{defi}
\newtheorem{definition}[theorem]{Definition}

\theoremstyle{remexample}

\theoremstyle{ass}

\newcommand{\Ecal}{\mathcal{E}}

\newcommand{\R}{\mathbb{R}}

\newcommand{\ONE}{\mathbbm{1}}

\def\XXint#1#2#3{{\setbox0=\hbox{$#1{#2#3}{\int}$} 
\vcenter{\hbox{$#2#3$}}\kern-.5\wd0}}

\newcommand{\p}{\partial}

\renewcommand{\phi}{\varphi}

\title{A toy model for the relativistic Vlasov-Maxwell system}

\author{Jonathan Ben-Artzi}

\address{School of  Mathematics, Cardiff University, Cardiff CF24 4AG, Wales, United Kingdom.}
\email{Ben-ArtziJ@cardiff.ac.uk}

\author{Stephen Pankavich}

\address{Department of Applied Mathematics and Statistics, Colorado School of Mines,
Golden, CO USA.}
\email{pankavic@mines.edu}

\author{Junyong Zhang}

\address{School of  Mathematics, Beijing Institute of Technology, Beijing, China; Cardiff University, Cardiff CF24 4AG, Wales, United Kingdom.}
\email{zhang\_junyong@bit.edu.cn; ZhangJ107@cardiff.ac.uk}

\thanks{The authors thank C. Pallard for explaining to them some of the delicate aspects of the Division Lemma, and the anonymous referees whose comments helped improve the presentation of the paper.}

\thanks{JBA  acknowledges  support  from  an  Engineering  and  Physical  Sciences  Research  Council  Fellowship (EP/N020154/1). SP acknowledges support from the US National Science Foundation under awards DMS-1911145 and DMS-2107938. JZ  acknowledges  support  from  the National  Natural  Science  Foundation  of  China (11771041, 11831004) and a Marie Sk\l odowska-Curie Fellowship (790623).}

\date{\today}



\begin{document}
\maketitle

\begin{abstract}
The global-in-time existence of classical solutions to the relativistic Vlasov-Maxwell (RVM) system in three space dimensions remains elusive after nearly four decades of mathematical research. In this note, a simplified ``toy model'' is presented and studied. This toy model retains one crucial aspect of the RVM system: the phase-space evolution of the distribution function is governed by a transport equation whose forcing term satisfies a wave equation with finite speed of propagation.\\
\medskip

\noindent \emph{The authors thank Claude Bardos and Fran\c{c}ois Golse who proposed this problem over dinner during the workshop ``The Cauchy Problem in Kinetic Theory: Recent Progress in Collisionless Models'' which was held at Imperial College London in 2015. That workshop was held in honor of Bob Glassey, to whose memory this paper is dedicated.}
\end{abstract}


\section{Introduction}\label{sec:intro}
Let $f(t,x,v)\geq0$ denote the one particle distribution in phase space of a monocharged plasma, where $x,v\in\R^3$ denote particle position and momentum, respectively, and $t\geq 0$ is the temporal variable. Taking relativistic effects into account, but neglecting collisions among the particles, $f$ satisfies the relativistic Vlasov-Maxwell system:
\begin{equation}
\tag{RVM}
\label{RVM}
\left. \begin{gathered}
\partial_{t}f+\hat{v} \cdot\nabla_{x}f+(E + \hat{v} \wedge B) \cdot\nabla_{v}f=0\\
\partial_{t} E=\nabla \wedge B- 4\pi j, \quad \nabla \cdot E=4\pi\rho,\\
\partial_{t} B=-\nabla \wedge E, \quad\nabla \cdot B=0,
\end{gathered} \right\}
\end{equation}
where
\begin{equation*}
\label{sources}
\rho(t,x)=\int_{\mathbb{R}^3} f(t,x,v)\,dv, \quad j(t,x)=\int_{\mathbb{R}^3} \hat{v} f(t,x,v)\,dv
\end{equation*}
are the charge and current density of the plasma, respectively, while
\begin{equation*}
\label{phat}
\hat{v} = \frac{v}{\sqrt{1 + \vert v \vert^2}}
\end{equation*}
is the relativistic velocity. Additionally, $E(t,x)$ and $B(t,x)$ are the self-consistent electric and magnetic fields generated by the charged particles, and we have chosen units such that the mass and charge of each particle, as well as the speed of light, are normalized to one.

The rigorous study of the relativistic Vlasov-Maxwell system largely dates back to the 1980s. A local-in-time existence and uniqueness result due to Wollman \cite{Wollman1984} was followed by the conditional result of Glassey and Strauss \cite{Glassey1986}, which to this day remains the most significant step toward a complete existence and uniqueness theory. In  \cite{Glassey1986} it is shown that solutions of \eqref{RVM} remain regular so long as one knows \emph{a priori} that particle momenta are uniformly bounded in time. In other words, if $\sup\{|v|\,|\,\exists x \in \mathbb{R}^3 \text{ s.t. }f(t,x,v)\neq0\}<+\infty$, then  the solution can be continued to time $t+h$ for some small $h>0$. This condition has been shown to hold for small \cite{Glassey1987a} and nearly neutral \cite{Glassey1988} data, and also in lower dimensions \cite{Glassey1990,Glassey1997}. Using Fourier methods, Klainerman and Staffilani \cite{Klainerman2002} provided an alternative method to prove the conditional result in \cite{Glassey1986}.  Bouchut, Golse and Pallard \cite{Bouchut2003} gave yet another proof which relied on the so-called ``division lemma'', which we use as well. More recently, Luk and Strain \cite{Luk2014} were able to improve the conditional result by weakening some of the assumptions.

The problem of global existence in three dimensions remains elusive. It is for this reason that attempts have been made to solve various toy models, in the hope that those may provide further insight into the full problem. Prior to our efforts, two related mean-field systems modeling resonance between a coupled wave equation and a transport equation have been investigated. In particular, G\'erard and Pallard \cite{Gerard2010} considered the one-dimensional relativistic problem
	\begin{equation*}
	\left.\begin{gathered}
	\partial_t f+\hat{v}\partial_x f+E\partial_v f=0,\\
	\Box E=\partial_x\rho,
	\end{gathered}\right\}
	\end{equation*}
where $\rho(t,x) = \int f(t,x,v) \ dv$ as in \eqref{RVM}.
Similarly, Nguyen and Pankavich \cite{Nguyen2014c} considered a related non-relativistic problem (also in one space and one momentum dimensions)
	\begin{equation*}
	\left.\begin{gathered}
	\partial_t f+{v}\partial_x f+B\partial_v f=0,\\
	(\partial_t+\partial_x)B=\rho,
	\end{gathered}\right\}
	\end{equation*} 
with each arriving at global existence results under limited assumptions.

The purpose of this note is to prove a global-in-time existence and uniqueness result for the following toy model of \eqref{RVM} kindly proposed to us by C. Bardos and F. Golse:
\begin{equation}\label{Toy}\tag{Toy}
\left.\begin{gathered}
\partial_t f+\hat{v}\cdot\nabla_x f- \partial_t A\cdot\nabla_v f=0,\\
\Box A=(\partial_{tt}-\Delta) A=j,
\end{gathered}\right\}
\end{equation} 
with initial data $f(0,x,v)=f_0(x,v)$ that is smooth and compactly supported and consistent data for $A$ satisfying $A(0,x)=A_0(x)$ and $\partial_t A(0,x)=A_1(x)$.
Here, the current density $j(t,x)$ is given by
\[j(t,x) = \int \hat{v} f(t,x,v) \ dv.\]
We note that \eqref{Toy} also couples a relativistic transport equation to a mean-field model of particle interaction given be a wave equation. For this system the position and momentum $x,v$ can be taken in $\R^d$ for any $d\geq1$, though $d=3$ is the primary case of interest. Our main result, similar to the two previous results, considers the  case $d=1$.

\subsection{Main results}
We prove local existence for bounded initial data and global existence for initial data that is once continuously differentiable and compactly supported.

\begin{theorem}[Local existence]\label{thm:local}
Suppose that $(f_0,A_0,A_1)\in W^{1,\infty}(\R^2)\times W^{1,\infty}(\R)\times L^\infty(\R)$ with compact support, then there exists $T>0$ such that 
the Cauchy problem \eqref{Toy} has a unique solution 
	\[
	(f,A)\in W^{1,\infty}([0,T)\times\R^2)\times W^{1,\infty}([0,T)\times\R).
	\]
If we denote the maximal lifespan of the solution by $T^*$, then $T^*<+\infty$ necessarily implies
\begin{equation*}
\limsup_{t\to T^*}\left(\|\partial_xf(t,x,v)\|_{L^\infty_{x,v}(\R^2)}+\|\partial_vf(t,x,v)\|_{L^\infty_{x,v}(\R^2)}\right)=+\infty.
\end{equation*}\end{theorem}

\begin{theorem}[Global existence]\label{thm:global}
If $(f_0,A_0,A_1)\in \mathcal{C}_c^1(\R^2)\times \mathcal{C}_c^1(\R)\times \mathcal{C}_c(\R)$, then the Cauchy problem \eqref{Toy} has a unique global solution such that
	\[(f,A)\in \mathcal{C}_c^1([0,\infty)\times\R^2)\times \mathcal{C}_c^1([0,\infty)\times\R).\]
\end{theorem}

The proofs of these theorems are contained within Sections \ref{sec:local} and \ref{sec:global}, respectively.
Next, we provide a justification for the structure of \eqref{Toy} and discuss the fundamental issues in obtaining analogous results in three dimensions.

\subsection{Justification of the toy model}
As determined by classical theory, the electromagnetic field $(E,B)$ in \eqref{RVM} is derived from potentials $\phi$ and $A$ that are given by
	\begin{align*}
	E=-\nabla\phi-\partial_tA,\qquad B=\nabla\wedge A.
	\end{align*}
In the Lorenz gauge, Maxwell's equations further reduce to the system of wave equations for the associated potentials, namely
	\begin{align*}
	\Box\phi=\rho,\qquad\Box A=j.
	\end{align*}
It is therefore evident that the simplified model \eqref{Toy} is obtained from \eqref{RVM} by neglecting the potential $\phi$ and assuming that $A$ is irrotational, i.e. $\nabla\wedge A=0$. These are not physically justifiable assumptions, yet they reduce \eqref{RVM} to a simplified system that still retains the main obstacle preventing us from proving global existence: the interplay between the Vlasov equation (which is a transport equation describing the evolution of particles in the system) whose speed of propagation has no \emph{a priori} bound, and the wave equations governing the fields which propagate at a constant and  finite speed (normalized to $c=1$ here).

An important feature of the system \eqref{Toy} is that it has a natural energy. While this conserved quantity is not used in the present note, it is an aspect of this toy model which makes it a natural  `sibling' of \eqref{RVM}. Indeed, multiplying the transport equation by $v_0:=\sqrt{1+|v|^2}$ and integrating in phase space, one easily finds that the quantity
	\begin{equation}
	\label{Ecal} \tag{Energy}
	\Ecal_{\text{Toy}}:=\iint v_0f(t,x,v)\ dv\ dx+\frac12\int\left(|\partial_tA|^2+|\nabla A|^2\right) dx
	\end{equation}
remains constant in time.
In particular, we note that performing this same operation within the previously studied toy problems does not appear to produce a conserved energy.


It is also illuminating to compare \eqref{Toy} with the Vlasov-Poisson system, which is the classical limit of the \eqref{RVM} system as the speed of light tends to infinity \cite{Schaeffer1986}:
\begin{equation}\label{eq:VP}\tag{VP}
\left.\begin{gathered}
\partial_t f+v\cdot\nabla_x f- \nabla\phi\cdot\nabla_v f=0,\\
-\Delta \phi=\rho.
\end{gathered}\right\}
\end{equation} 
The two systems -- \eqref{eq:VP} and \eqref{Toy} -- are very similar, though the former features classical transport, while the latter is relativistic. The most important distinction arises in the equation for the potential, which is elliptic  in \eqref{eq:VP} and  hyperbolic  in \eqref{Toy}.
Additionally, we note that the conserved energy is a crucial ingredient within some of the known proofs \cite{Pfaffelmoser1992, Schaeffer1991}   of global-in-time existence for smooth solutions of \eqref{eq:VP}.
Hence, one may expect \eqref{Ecal} to be similarly important to the study of \eqref{Toy} in three dimensions.

\section{Local existence and uniqueness}\label{sec:local}
We first establish the local-in-time existence and uniqueness result using a fixed-point argument similar to\cite{Gerard2010}.
\begin{proof}[Proof of Theorem \ref{thm:local}]
Since the initial data is assumed to have compact support, we can fix $R>0$ and $M>0$ such that $f_0\in \mathcal{C}_c((-R,R)\times(-M,M))$.

\begin{definition}[The Set $B_T$]
For a given $T>0$, we define $B_T$ to be the set of functions $g \in W^{1,\infty}([0,T]\times\R^2)$ that satisfy
\vspace{0.2cm}

(H1) $g(0,x,v)=f_0(x,v)$ and $\|g\|_{L^\infty([0,T]\times\R^2)}\leq \|f_0\|_{W^{1,\infty}(\R^2)}$;

(H2) $\text{supp}~g\subset [0,T]\times (-R-1, R+1)\times(-M-1,M+1)$;

(H3) $\|g\|_{\text{Lip}}\leq 3\|f_0\|_{W^{1,\infty}(\R^2)}$, where
\begin{equation*}
\|g\|_{\text{Lip}}:=\sup_{\substack{t \in [0,T] \\ x,v\in\R \\ h=(h_1,h_2) \neq 0}}\frac{\left|g(t,x+h_1,v+h_2)-g(t,x,v)\right|}{|h|};
\end{equation*}


(H4) $\|\partial_t g\|_{L^\infty([0,T]\times\R^2)}\leq 3\|f_0\|_{W^{1,\infty}(\R^2)}(2+\|A_0\|_{L^\infty(\R)}+\|A_1\|_{L^\infty(\R)})$.
\end{definition}
\vspace{0.2cm}

When endowed with the metric $d(g_1,g_2):=\|g_1-g_2\|_{L^\infty([0,T]\times\R^2)}$, the metric space $(B_T,d)$ is complete. Next, for any given $g\in B_T$, we define $A_g$ to be the solution to the linear wave equation
\begin{equation}\label{W}
\Box A_g=(\partial_{t}^2-\partial_x^2) A_g=\int_{\R} \hat{v} g(t,x,v) dv:=j_g(t,x),
\end{equation} 
with initial conditions $A_g(0,x)=A_0(x)$ and $\partial_t A_g(0,x)=A_1(x)$ for any $x \in \R$.
\begin{definition}[The Solution Map $\Phi$]
For any $g \in B_T$ we define  the solution map $f =\Phi(g)$ where $f \in W^{1,\infty}((0,T)\times\R^2)$ to be the unique solution of the transport equation
\begin{equation}\label{T}
\partial_t f+\hat{v}\partial_x f- \partial_t A_g(t,x)\partial_v f=0,\quad (t,x,v)\in(0,T)\times\R^2,
\end{equation} 
with initial condition $f(0,x,v)=f_0(x,v)$.
\end{definition}
For the fixed-point argument it suffices to show the following two properties hold for $T$ sufficiently small:

(1) $\Phi$ maps $B_T$ into itself, i.e. for every $g\in B_T$,  $f=\Phi(g)\in B_T$;

(2) $\Phi: B_T\to B_T$ is a contraction, i.e. there is $0<C<1$ such that 
$$\| \Phi(g_1) - \Phi(g_2) \|_{L^\infty([0,T]\times\R^2)} \leq C \|g_1 - g_2\|_{L^\infty([0,T]\times\R^2)}$$
for every $g_1, g_2 \in B_T$. \\

\textbf{Step 1: Preliminary estimates.} Throughout we will use the fact that the mapping $v \mapsto \hat{v} = v/\sqrt{1+ v^2}$ and its derivative are both bounded above by one. The function $A_g$ of \eqref{W} is given by the solution of the wave equation, namely
\begin{equation}\label{wave-s1}
\begin{split}
 A_g(t,x)&=(\partial_t Y(t,\cdot)*_x A_0)(t,x)+(Y(t,\cdot)*_x A_1)(t,x)+(Y(\cdot,\cdot)*_{t,x} (j_g\ONE_{t>0}(t))(t,x)
\end{split}
\end{equation} 
where $Y(t,x)=\frac12 \ONE_{\{|x|\leq t\}}$ is the forward fundamental solution of the one-dimensional wave operator.
Note that the derivatives of $Y$ satisfy
	\begin{equation}\label{eq:Y-deriv}
	\partial_x Y(t,x)
	=
	\frac12\delta_{x=-t}-\frac12\delta_{x=t},
	\quad
	\partial_t Y(t,x)
	=
	\frac12\delta_{x=-t}+\frac12\delta_{x=t}
	\end{equation}
(these expressions will be used later). More explicitly, the d'Alembert formula gives
	\begin{equation*}\label{wave-s2}
	A_g(t,x)
	=
	\frac{1}2[A_0(x+t)+A_0(x-t)]+\frac12\int_{x-t}^{x+t} A_1(s)\,ds+\frac12\int_0^t\int_{x-(t-s)}^{x+(t-s)}j_g(s,y)\,dy\, ds.
	\end{equation*} 
Thus, for $t\in [0,T]$ we find
\begin{align}
 \|A_g(t,\cdot)\|_{L^\infty(\R)}&\leq \|A_0\|_{L^\infty(\R)}+t\|A_1\|_{L^\infty(\R)}+\frac{1}{2}t^2\|j_g\|_{L^\infty([0,T]\times\R)},\notag\\
  \|\partial_tA_g(t,\cdot)\|_{L^\infty(\R)}&\leq \|A'_0\|_{L^\infty(\R)}+\|A_1\|_{L^\infty(\R)}+t\|j_g\|_{L^\infty([0,T]\times\R)}\label{A_g-estimates},\\
    \|\partial_x\partial_tA_g(t,\cdot)\|_{L^\infty(\R)}&\leq \|A''_0\|_{L^\infty(\R)}+\|A'_1\|_{L^\infty(\R)}+t\|\partial_x j_g\|_{L^\infty([0,T]\times\R)}.\notag
\end{align} 
Taking $g\in B_T$, we have by (H1), (H2) and (H3)
\begin{align*}
&\|j_g\|_{L^\infty([0,T]\times\R)}= \left\| \int_{\R} \hat{v} g(\cdot,\cdot,v)dv\right\|_{L^\infty([0,T]\times\R)}\leq 2(M+1)\|f_0\|_{W^{1,\infty}(\R^2)},\\
&\|\partial_x j_g\|_{L^\infty([0,T]\times\R)}= \left\| \int_{\R} \hat{v} \partial_x g(\cdot,\cdot,v)dv\right\|_{L^\infty([0,T]\times\R)}\leq 6(M+1)\|f_0\|_{W^{1,\infty}(\R^2)}.
\end{align*}
Hence, taking $T$ sufficiently small we obtain the estimates
\begin{align}
 \|A_g\|_{L^\infty([0,T]\times\R)}&\leq \|A_0\|_{L^\infty(\R)}+\|A_1\|_{L^\infty(\R)}+1,\notag\\
  \|\partial_t A_g\|_{L^\infty([0,T]\times\R)}&\leq \|A_0'\|_{L^\infty(\R)}+\|A_1\|_{L^\infty(\R)}+1, \label{Ab}\\
  \|\partial_x \partial_t A_g\|_{L^\infty([0,T]\times\R)}&\leq \|A_0''\|_{L^\infty(\R)}+\|A_1'\|_{L^\infty(\R)}+1. \label{Ac}
\end{align} 

\textbf{Step 2:   $\Phi$ maps  $B_T$ into itself.} 
Denote by $(X(s;t,x,v), V(s;t,x,v))$ the characteristic curves of \eqref{T}. They satisfy the system of ODEs
\begin{equation}\label{eq:char-odes}
\left .
\begin{split}
&\frac{d X}{ds}(s; t, x, v)=\hat{V}(s; t, x, v)=\frac{V(s;t,x,v)}{\sqrt{1+V^2(s;t,x,v)}},\\
&\frac{d V}{ds}(s; t, x, v)=-(\partial_t A_g)(s, X(s; t,  x, v)),
\end{split}
\right \}
\end{equation} 
with the initial conditions $X(t;t,x,v)=x$ and  $V(t;t,x,v)=v$. It is well-known that the solution of the transport equation \eqref{T} can  be expressed as
\begin{equation*}
f(t,x,v)=f_0(X(0;t,x,v),V(0;t,x,v)).
\end{equation*}

From this, we immediately find that $f =\Phi(g)$ satisfies $f(0,x,v)=f_0(x,v)$ and 
\begin{equation*}
\|f\|_{L^\infty([0,T])\times\R^2)}=\|f_0(X(0;\cdot,\cdot,\cdot),V(0;\cdot,\cdot,\cdot)))\|_{L^\infty([0,T])\times\R^2)}=\|f_0\|_{L^\infty(\R^2)}
\end{equation*}
so that (H1) is satisfied.
%
Additionally, using \eqref{Ab} we have
for $T>0$ sufficiently small
\begin{align*}
|x|=|X(t;t,x,v)|\leq |X(0;t,x,v)|+ \int_0^t |\hat{V}(s; t, x, v)| ds\leq |X(0;t,x,v)|+ 1
\end{align*}
and
\begin{align*}
|v|=|V(t;t,x,v)|\leq |V(0;t,x,v)|+\int_0^t |(\partial_t A_g)(s, X(s; t,  x, v))| ds\leq |V(0;t,x,v)|+1.
\end{align*} 
for $t\in[0,T]$. 
Now, for $(x,v)\in \text{supp}(f(t,\cdot,\cdot))$  one has  $0\neq f(t,x,v)=f_0(X(0;t,x,v),V(0;t,x,v))$ and consequently $|X(0;t,x,v)|\leq R,
|V(0;t,x,v)|\leq M$. Therefore, $|x|\leq R+1$ and $|v|\leq M+1$ and (H2) is satisfied.\\

Next, we verify (H3).
By the definition of the Lipschitz norm, we have 
	\begin{align*}
	\|f\|_{\text{Lip}}
	&=
	\sup_{\substack{t \in [0,T]\\ (x,v)\neq(y,p) \in \R^2}}\frac{|f(t,x,v)-f(t,y,p)|}{|(x-y,v-p)|}\\
	&=
	\sup_{\substack{t \in [0,T]\\ (x,v)\neq(y,p) \in \R^2}}\frac{|f_0(X(0;t,x,v),V(0;t,x,v))-f_0(X(0;t,y,p),V(0;t,y,p))|}{|(x-y,v-p)|}\\
	&\leq \|f_0\|_{W^{1,\infty}(\R^2)}(\|X\|_{\text{Lip}}+\|V\|_{\text{Lip}})
	\end{align*} 
where the Lipschitz norms of the characteristics are defined by
\begin{equation*}
\|X\|_{\text{Lip}}:=\sup_{\substack{s,t \in [0,T] \\ x,v\in\R \\ h=(h_1,h_2) \neq 0}}\frac{\left |X(s;t,x+h_1,v+h_2)-X(s;t,x,v)\right|}{|h|}
\end{equation*}
and analogously for $\|V\|_{\text{Lip}}$.
Integrating the characteristics of \eqref{eq:char-odes} yields
\begin{align*}
X(\tau;t,x,v) &=x+ \int_t^\tau \hat{V}(s; t, x, v) ds,\\ 
V(\tau;t,x,v) &=v-\int_t^\tau (\partial_t A_g)(s, X(s; t,  x, v)) ds,
\end{align*} 
which provides the following bounds on the Lipschitz norms:
\begin{align*}
\|X\|_{\text{Lip}}&\leq 1+T\|V\|_{\text{Lip}},\\
 \|V\|_{\text{Lip}}&\leq 1+T\|\partial_x\partial_t A_g\|_{L^\infty([0,T]\times\R)}\|X\|_{\text{Lip}}.
\end{align*} 
Summing and using \eqref{Ac} then gives 
\begin{equation*}
\|X\|_{\text{Lip}}+ \|V\|_{\text{Lip}}\leq 2+\frac13\left(\|V\|_{\text{Lip}}+\|X\|_{\text{Lip}}\right),
\end{equation*} 
for $T$ sufficiently small, which implies
$\|X\|_{\text{Lip}}+ \|V\|_{\text{Lip}}\leq 3$.
Inserting this into the estimate on $\| f\|_{\text{Lip}}$, we conclude
\begin{equation}
\label{eq:f-lip-norm}
\|f\|_{\text{Lip}}\leq 3 \|f_0\|_{W^{1,\infty}(\R^2)}
\end{equation}
and (H3) is satisfied.

Finally, we verify (H4). Computing the time derivative of $f$, we find
\begin{align*}
\|\partial_t f\|_{L^\infty([0,T]\times\R^2)}&=\|\partial_t[ f_0(X(0;\cdot,\cdot,\cdot),V(0;\cdot,\cdot,\cdot))]\|_{L^\infty([0,T]\times\R^2)}\\
&\leq \|f_0\|_{W^{1,\infty}(\R^2)} \left (|\partial_t X(0;\cdot, \cdot, \cdot)|_{L^\infty([0,T]\times\R^2)}+|\partial_t V(0;\cdot, \cdot, \cdot)|_{L^\infty([0,T]\times\R^2)} \right ).
\end{align*}
To bound the two terms on the right hand side, we first estimate
\begin{align*}
|\partial_t X(\tau;t,x,v)|&=\left|\partial_t\left(x+ \int_t^\tau \hat{V}(s; t, x, v) ds\right)\right|\\
&=\left|-\hat{v}- \int_\tau^t \partial_v (\hat{V}(s; t, x, v)) \partial_t V(s;t,x,v)ds\right|\\
&\leq 1+ \int_0^t |\partial_t V(s;t,x,v)| ds
\end{align*}
and
\begin{align*}
\left|\partial_tV(\tau;t,x,v)\right|
&=\left|\partial_t\left(v-\int_t^\tau (\partial_t A_g)(s, X(s; t,  x, v)) ds\right)\right|\\
&=\left|\partial_tA_g(t, x) + \int_\tau^t (\partial_x\partial_t A_g)(s, X(s; t,  x, v))  \partial_t X(s;t,x,v)ds\right|\\
&\leq
\|\partial_tA_g\|_{L^\infty([0,T]\times\R)} + \|\partial_x\partial_t A_g\|_{L^\infty([0,T]\times\R)} \int_0^t |\partial_t X(s;t,x,v)|ds.
\end{align*} 
Therefore, using \eqref{Ab} we obtain
\begin{align*}
\sup_{\tau \in [0,t]} \left (|\partial_t X(\tau,t,x,v)| \right .&+ \left.|\partial_t V(\tau,t,x,v)| \right )
\leq 1 +\|\partial_tA_g\|_{L^\infty([0,T]\times\R)} \\
&\hspace{-1cm} + (1 + \|\partial_x\partial_t A_g\|_{L^\infty([0,T]\times\R)}) \int_0^t (|\partial_t V(s;t,x,v)|+|\partial_t X(s;t,x,v)|) ds\\
&\hspace{-2cm}\leq  2+ \|A'_0\|_{L^\infty(\R)}+\|A_1\|_{L^\infty(\R)}\\
& \hspace{-1cm}+ ( 2+\|A''_0\|_{L^\infty(\R)}+\|A'_1\|_{L^\infty(\R)}) \int_0^t \sup_{\tau \in [0,s]} \left (|\partial_t X(\tau;t,x,v)| + |\partial_t V(\tau;t,x,v)| \right) ds.
\end{align*} 
Invoking Gr\"onwall's inequality now yields
\begin{align*}
\sup_{\tau \in [0,t]} \left (|\partial_t X(\tau;t,x,v)| \right.+& \left.|\partial_tV(\tau;t,x,v)| \right )\\
\leq
&\left(2+ \|A'_0\|_{L^\infty(\R)}+\|A_1\|_{L^\infty(\R)} \right)e^{t(2+\|A''_0\|_{L^\infty(\R)}+\|A'_1\|_{L^\infty(\R)})}.
\end{align*} 
Using this estimate we ultimately find
$$|\partial_t X(0;t,x,v)| + |\partial_tV(0;t,x,v)| \leq  3 \left (2+ \|A'_0\|_{L^\infty(\R)}+\|A_1\|_{L^\infty(\R)} \right )$$
for all $t \in [0,T]$, $x,v \in \R$ and $T > 0$ sufficiently small.
Taking the supremum over $x,v\in \mathbb{R}$ and combining this with the estimate of $\| \partial_t f \|_{L^\infty([0,T]\times\R^2)}$ yields
$$\|\partial_t f\|_{L^\infty([0,T]\times\R^2)}\leq 3\|f_0\|_{W^{1,\infty}(\R^2)}(2+\|A_0\|_{L^\infty(\R)}+\|A_1\|_{L^\infty(\R)})$$
and (H4) is satisfied.\\

\textbf{Step 3: $\Phi$ is a contraction.}
Let $g,\tilde{g}\in B_T$ and $f=\Phi(g), \tilde{f}=\Phi(\tilde{g})$. Then, subtracting the respective Vlasov equations yields
\begin{equation*}\label{T'}
\partial_t (f-\tilde{f})+\hat{v}\partial_x (f-\tilde{f})- \partial_t A_g(t,x)\partial_v (f-\tilde{f})- \big(\partial_t A_g(t,x)-\partial_t A_{\tilde g}(t,x)\big)\partial_v \tilde{f}=0
\end{equation*} 
with $(f-\tilde f)(0,x,v)=0$. Consequently
\begin{align}
& (f-\tilde{f})(t, X(t;0,x,v), V(t;0,x,v))\notag
\\&=-\int_0^t \big(\partial_t A_g-\partial_t A_{\tilde g}\big)(s, X(s;0,x,v))\cdot \partial_v \tilde{f}(s,X(s;0,x,v), V(s;0,x,v))ds.\label{eq:f-ftilde}
\end{align} 
Using \eqref{A_g-estimates} and \eqref{eq:f-lip-norm}, we have the following estimates for the right side of \eqref{eq:f-ftilde}:
\begin{align*}
&\|\partial_v \tilde{f}\|_{L^\infty([0,T]\times\R^2)}\leq 3 \|f_0\|_{W^{1,\infty}(\R^2)},\\
&\|\partial_t A_g-\partial_t A_{\tilde g}\|_{L^\infty([0,T]\times\R)}\leq 2T\|j_{g}-j_{\tilde{g}}\|_{L^\infty([0,T]\times\R)}.
\end{align*} 
Therefore, we find from \eqref{eq:f-ftilde} that 
\begin{equation*}
\|f-\tilde{f}\|_{L^\infty([0,T]\times\R^2)}\leq 3 \|f_0\|_{W^{1,\infty}(\R^2)} T^2 \|j_{g}-j_{\tilde{g}}\|_{L^\infty([0,T]\times\R)},
\end{equation*}
which implies
\begin{equation*}
\|f-\tilde{f}\|_{L^\infty([0,T]\times\R^2)}\leq \frac12\|g-\tilde{g}\|_{L^\infty([0,T]\times\R^2)}
\end{equation*}
provided that $T$ is sufficiently small.
Thus, we obtain a unique local solution to the Cauchy problem \eqref{Toy} on $[0,T]$ for $T$ sufficiently small.
Furthermore, we can extend the lifespan of the solution as long as derivatives remain finite, namely for any $ t\in[0,T]$ such that
\begin{equation*}
\|\partial_xf(t,\cdot,\cdot)\|_{L^\infty(\R^2)}+\|\partial_vf(t,\cdot,\cdot)\|_{L^\infty(\R^2)}<+\infty.
\end{equation*}
This completes the proof.
\end{proof}

\section{Global existence}\label{sec:global}
With the existence of a local-in-time solution established, we now extend the solution globally by uniformly bounding the momentum support of the distribution function and the derivatives of the field.
\begin{proof}[Proof of Theorem \ref{thm:global}]
We assume that the maximal life span is $[0, T^*)$ for some $T^*>0$, and shall  prove that $T^*=+\infty$, hence
the  solution is global. We need to show
	\begin{equation}\label{eq:lifespan-cond}
	(f,\partial_t A)\in W^{1,\infty}([0, T^*)\times \R^2)\times W^{1,\infty}([0, T^*)\times \R).\\
	\end{equation}

\medskip
\textbf{Step 1: Bounds on $f$ and $\p_tA$.}
From the proof of local existence, we know that $f$ has compact support for any $t\in [0, T^*)$.  In particular, the $v$ support of $f$ is uniformly bounded for any fixed time $t\in [0,T^*)$ but may tend to $+\infty$ as $t\to T^*$. We therefore define the following crucial quantity
\begin{equation*}
P(t):=\sup\{|v|: \exists \,x \in \mathbb{R} \, \text{such that}\, f(t,x,v)\neq 0\}
\end{equation*}
and prove the following result.
\begin{proposition}\label{prop:M-bound}
Let $T^*$ be the maximal lifespan of the solution and let $T\in(0,T^*)$. Then there exists $C>0$ independent of $T$ such that
\begin{equation}\label{eq:A-bound}
\|\partial_t A\|_{L^\infty([0,T] \times \R)}\leq C
\end{equation}
and
	\begin{equation}\label{eq:P-bound}
	P(T)\leq C.
	\end{equation}
\end{proposition}
\begin{proof}
Our strategy is classical: we generate a Gr\"onwall-type inequality by first using the wave equation $\Box A=j$ to show that $\partial_tA$ is controlled by $j$, then showing that $j$ can be controlled by $P(t)$, and finally bounding $P(t)$ by a time-integral of $\partial_tA$.

 Define the functions
\begin{align*}
B^\pm(t,x)&=\partial_t A(t,x)\pm\partial_x A(t,x).
\end{align*} 
Using the relationship $\Box A=j$, we have
\begin{align*}
(\partial_t\mp\partial_x)B^\pm(t,x)&=j(t,x),
\end{align*} 
so that, for any $h>0$,
\begin{align*}
\partial_\tau[B^\pm(\tau,x\pm(t+h-\tau))]&=[(\partial_t\mp\partial_x)B^\pm](\tau,x\pm(t+h-\tau))=j(\tau,x\pm(t+h-\tau)).
\end{align*} 
Integrating with respect to $\tau \in [t,t+h]$, we obtain
\begin{align*}
B^\pm(t+h,x)&=B^\pm(t,x\pm h)+\int_t^{t+h}j(\tau,x\pm(t+h-\tau)) d\tau.
\end{align*} 
Taking $t=0$ and replacing $h$ by $t$, we can represent $B^\pm$ as: 
\begin{align*}
B^\pm(t,x)&=B^\pm(0,x\pm t)+\int_0^{t}j(\tau,x\pm(t-\tau)) d\tau.
\end{align*} 
This allows us to represent $\p_tA$ as follows:
	\begin{align}
	\partial_tA(t,x)
	=&
	\frac12(B^+(t,x)+B^-(t,x))\notag\\
	=&
	\frac 12\left(A_0'(x+t)-A_0'(x-t)+A_1(x+t)+A_1(x-t)\right)\label{paA}\\
	&+
	\frac12\int_0^{t}\left(j(\tau,x-(t-\tau))+j(\tau,x+(t-\tau))\right) d\tau.\notag
	\end{align}

We are now ready to prove \eqref{eq:A-bound}. Using \eqref{paA} and the properties of the initial data, for $T\in[0,T^*)$ we can estimate 
	\begin{align*}
	\|\partial_t A\|_{L^\infty([0,T] \times\R)}
	&\leq
	C_1\left(1+\sum_\pm\sup_{\substack{x\in\R\\ t\in[0, T]}}\left|\int_0^{t}j(\tau,x\pm(t-\tau)) d\tau \right|\right)
	\end{align*} 
where $C_1$  only depends on the initial data.
We therefore turn to bounding $\|j\|_{L^\infty([0,T]\times\R)}$. Because the relativistic velocity is bounded above by $|\hat{v}|<1$ and using the definition of $P(t)$, we have
	\begin{align*}
	|j(t,x)|
	=
	\left|\int_\R\hat{v}f(t,x,v)dv\right|
	\leq
	\|f_0\|_{L^\infty}P(t)
	\end{align*}
for all $t \in [0,T]$ and $x \in \mathbb{R}$.	
Due to the characteristic equations \eqref{eq:char-odes}, the change in velocity is governed by $\partial_tA$ so that
	\begin{equation*}
	P(t)
	\leq
	P(0)+\int_0^t\|\partial_tA(\tau,\cdot)\|_{L^\infty}d\tau.
	\end{equation*}
Inserting the last three estimates into one another, we find
	\begin{equation}\label{eq:gronwall}
	\|\partial_t A\|_{L^\infty([0,T] \times\R)}
	\leq
	C_2\left(1+\int_0^T\|\partial_tA(\tau,\cdot)\|_{L^\infty([0,T] \times\R)} d\tau\right)
	\end{equation}
where $C_2$ also only depends on the initial data.
A standard Gr\"onwall argument applied to \eqref{eq:gronwall} yields
	\begin{equation*}
	\|\partial_t A\|_{L^\infty([0,T] \times\R)}
	\leq
	C_3e^T
	\leq
	C_3e^{T^*}
	\leq
	C
	\end{equation*}
where $C_3$ again only depends on the initial data and $C<+\infty$ depends only on the initial data and $T^*$, but not on $T$. Therefore
	\begin{equation*} 
	P(T)
	\leq
	P(0)+CT^*.
	\end{equation*}
This completes the proof of Proposition \ref{prop:M-bound}.
\end{proof}

\textbf{Step 2: Bounds on the derivatives of $f$ and $\p_tA$.}
The transport  equation  for $f$ in \eqref{Toy} takes the following form for the derivatives of $f$:
\begin{equation*}
(\partial_t +\hat{v}\partial_x + \partial_t A(t,x)\partial_v) \left(\begin{matrix}\partial_x f\\ \partial_v f\end{matrix}\right)=
-\left(\begin{matrix}0 &\partial_x\partial_t A(t,x)\\ (1+|v|^2)^{-3/2}& 0\end{matrix}\right)\left(\begin{matrix}\partial_x f\\ \partial_v f\end{matrix}\right).
\end{equation*} 

We therefore need to bound $\partial_x\partial_t A$.

\begin{proposition}\label{prop:bound-Atx}
Let $T^*$ be the maximal lifespan of the solution and let $T\in(0,T^*)$. Then there exists $C>0$ independent of $T$ such that
\begin{equation*}
\|\partial_x\partial_t A\|_{L^\infty([0,T] \times \R)}\leq C.
\end{equation*}
\end{proposition}

\begin{proof}
Since $A$ satisfies $\Box A=j$, inverting the wave operator means that $A$ is obtained from $j$ and the initial data via the expression \eqref{wave-s1}. Assuming, without loss of generality, that the initial data for the field is trivial, i.e., $A_0=A_1=0$,  \eqref{wave-s1} reduces to
\begin{equation*}
 A_f(t,x)=(Y(\cdot,\cdot)*_{t,x} (j_f\ONE_{t>0})(t,x),
\end{equation*} 
where we recall that $Y=\frac12 \ONE_{\{|x|\leq t\}}$ is the forward fundamental solution of the one-dimensional wave operator. Therefore
	\begin{align*}
	\partial_x\partial_t A_f(t,x)
	=\partial_x\partial_t\left(Y(\cdot,\cdot)*_{t,x} (j_f\ONE_{t>0})\right)(t,x)
	=\left(\partial_tY(\cdot,\cdot)*_{t,x} (j_{\partial_x f}\ONE_{t>0})\right)(t,x),
	\end{align*}
and using the Vlasov equation $\partial_t f+\hat{v}\partial_x f-\partial_t A_f \partial_{v}f=0$ in the term $j_{\partial_x f}=\int\hat{v}\partial_x f\,dv$, we have
	\begin{align*}
	\partial_x\partial_t A_f(t,x)
	&=\left[\partial_tY(\cdot,\cdot)*_{t,x} \left(\partial_t \rho_f \ONE_{t>0}\right)\right](t,x).
	\end{align*}
Integrating by parts in the convolution (and henceforth dropping the subscript $f$ for brevity), we transfer the time derivative from $\rho$ to $Y$ so that 
	\begin{align*}
	\partial_x\partial_t A(t,x)
	&=\left[\partial_{tt}Y(\cdot,\cdot)*_{t,x}\rho \ONE_{t>0}\right](t,x)+\left[\partial_{t}Y(t,\cdot)*_x\rho(0,\cdot)\right](x). 
	\end{align*} 
As the fundamental solution $Y$ satisfies $\Box Y=\delta_{(t,x)=(0,0)}$, this further yields
  	\begin{align*}
	\partial_x\partial_t A(t,x)
	&=
	\left[\left(\partial_{xx}Y+\delta_{(t,x)=(0,0)}\right)*_{t,x}\rho\ONE_{t>0}\right](t,x)+\left[\partial_{t}Y(t,\cdot)*_x\rho(0,\cdot)\right](x)\\
	&=
	\underbrace{\left[\partial_{xx}Y*_{t,x}\rho\ONE_{t>0}\right](t,x)}_{I}+\underbrace{\left[\partial_{t}Y(t,\cdot)*_x\rho(0,\cdot)\right](x)}_{II}+\rho(t,x).
	\end{align*} 
Let us consider the  term $I$. Using the division lemma (Lemma \ref{lem:division}) with $a(v)=\hat{v}=\frac{v}{\sqrt{1+v^2}}$ we write $\partial_{xx}Y$ as
	\begin{equation*}
	\partial_{xx}Y
	=
	(\partial_t+\hat{v}\partial_x)\left(\frac{x}{\hat{v}x-t}\partial_x Y\right)+(1+v^2)\delta_{(t,x)=(0,0)}.
	\end{equation*}
Because this holds for every $v$, in the term $I$ we replace $\rho$ with $\int f\,dv$ and get
	\begin{align*}
	I
	&=
	\int\left[(\partial_t+\hat{v}\partial_x)\left(\frac{x}{\hat{v}x-t}\partial_x Y\right)*_{t,x}f(\cdot,\cdot,v)\ONE_{t>0}\right](t,x)\,dv\\
	&\qquad+
	\int\left[(1+v^2)\delta_{(t,x)=(0,0)}*_{t,x}f(\cdot,\cdot,v)\ONE_{t>0}\right](t,x)\,dv\\
	&=
	\int\left[\frac{x}{\hat{v}x-t}\partial_x Y*_{t,x}(\partial_t+\hat{v}\partial_x)(f(\cdot,\cdot,v)\ONE_{t>0})\right](t,x)\,dv+\int(1+v^2)f(t,x,v)\,dv\\
	&=
	\underbrace{\int\left[\frac{x}{\hat{v}x-t}\partial_x Y*_{t,x}\partial_tA\partial_vf(\cdot,\cdot,v)\ONE_{t>0}\right](t,x)\,dv}_{I_a}\\
	&\qquad+
	\underbrace{\int\left[\frac{x}{\hat{v}x-t}\partial_x Y*_{t,x}f(\cdot,\cdot,v)\delta_{t=0}\right](t,x)\,dv}_{I_b}
	+
	\int(1+v^2)f(t,x,v)\,dv.
	\end{align*}
Using the properties of the derivatives of $Y$ (see \eqref{eq:Y-deriv}), we can simplify the terms $I_a$ and $I_b$. Let us  first consider the term $I_a$:
	\begin{align*}
	I_a
	&=
	\int\left[\frac{x}{\hat{v}x-t}\partial_x Y*_{t,x}\partial_tA\partial_vf(\cdot,\cdot,v)\ONE_{t>0}\right](t,x)\,dv\\
	&=
	\frac12\int\left[\frac{x}{\hat{v}x-t}(\delta_{x=-t}-\delta_{x=t})*_{t,x}\partial_tA\partial_vf(\cdot,\cdot,v)\ONE_{t>0}\right](t,x)\,dv\\
	&=
	\frac12\int\int_0^t\int\frac{y}{\hat{v}y-s}(\delta_{y=-s}-\delta_{y=s})\partial_tA(t-s,x-y)\partial_vf(t-s,x-y,v)\,dy\,ds\,dv\\
	&=
	\frac12\sum_\pm\int\int_0^t\frac{1}{1 \pm \hat{v}}\partial_tA(t-s,x\pm s)\partial_vf(t-s,x\pm s,v)\,ds\,dv.
	\end{align*}
To integrate by parts in $v$, we observe that
	\begin{equation*}
	\frac{d}{dv}\left(\frac{1}{1\pm\hat{v}}\right)
	=
	\frac{v\mp \sqrt{1+v^2}}{1+v^2\pm v\sqrt{1+v^2}}
	\end{equation*}
and we obtain
	\begin{align*}
	I_a
	&=
	\frac12\sum_\pm\mp\int\int_0^t\frac{v\mp \sqrt{1+v^2}}{1+v^2\pm v\sqrt{1+v^2}}\partial_tA(t-s,x\pm s)f(t-s,x\pm s,v)\,ds\,dv.
	\end{align*}
Turning to the term $I_b$, and again using \eqref{eq:Y-deriv}, we have
	\begin{align*}
	I_b
	&=
	\frac12\int\left[\frac{x}{\hat{v}x-t}(\delta_{x=-t}-\delta_{x=t})*_{t,x}f(\cdot,\cdot,v)\delta_{t=0}\right](t,x)\,dv\\
	&=
	\frac12\int\int_0^t\int\frac{y}{\hat{v}y-s}(\delta_{y=-s}-\delta_{y=s})f(t-s,x-y,v)\delta_{t-s=0}\,dy\,ds\,dv\\
	&=
	\frac12\sum_\pm\pm\int\frac{1}{\hat{v}\pm1}f(0,x\pm t,v)\,dv.
	\end{align*}
Now we can consider the term $II$ which easily simplifies to
	\begin{align*}
	II
	=
	\left[\partial_{t}Y(t,\cdot)*_x\rho(0,\cdot)\right](x)
	=
	\frac12\int(\delta_{y=-t}+\delta_{y=t})\rho(0,x-y)\,dx
	=
	\frac12\sum_\pm\rho(0,x\pm t).
	\end{align*}
Collecting these terms, we have
	\begin{align*}
	\partial_x\partial_t A(t,x)
	&=
	\frac12\sum_\pm\mp\int\int_0^t\frac{v\mp \sqrt{1+v^2}}{1+v^2\pm v\sqrt{1+v^2}}\partial_tA(t-s,x\pm s)f(t-s,x\pm s,v)\,ds\,dv\\
	&\qquad+
	\underbrace{\frac12\sum_\pm\pm\int\frac{1}{\hat{v}\pm1}f(0,x\pm t,v)\,dv+
	\frac12\sum_\pm\rho(0,x\pm t)}_{\text{data}}\\
	&\qquad+\int(2+v^2)f(t,x,v)\,dv.
	\end{align*}
To estimate $\partial_x\partial_t A_f(t,x)$ we need to bound the first and last terms; the ``data'' terms depend on the initial data and are therefore finite and independent of $t$.
The integrand within the first term can be controlled by $1+ P(T)$.
Indeed, letting $v_0 = \sqrt{1 + v^2}$ yields
$$\left|\frac{v\mp \sqrt{1+v^2}}{1+v^2\pm v\sqrt{1+v^2}}\right| = v_0^{-1} \left | \frac{v \mp v_0}{v_0 \pm v} \right | = v_0^{-1} |v_0 \mp v |^2 \leq 2v_0,$$
and thus
\begin{align*}
\sup_{|v|\leq P(T)}\left|\frac{v\mp \sqrt{1+v^2}}{1+v^2\pm v\sqrt{1+v^2}}\right|
\leq 2 \sqrt{1 + P(T)^2}
\leq 2(1+ P(T)).
\end{align*}

Hence, allowing $C$ to be a constant independent of $T$ that may change from line to line, we have
	\begin{align*}
	\|\partial_x\partial_t A&\|_{L^\infty([0,T] \times\R)}
	\leq
	\text{data}\\
	&+
	CTP(T)\|\partial_t A\|_{L^\infty([0,T] \times\R)}\|f_0\|_{L^\infty(\R)}\sup_{|v|\leq P(T)}\left|\frac{v\mp \sqrt{1+v^2}}{1+v^2\pm v\sqrt{1+v^2}}\right|\\
	&+
	\|f_0\|_{L^\infty(\R)}(2+P(T)^2)P(T)\\
	\leq&
	\text{data}+CP(T)\|f_0\|_{L^\infty(\R)}\left(T^*(1+P(T))\|\partial_t A\|_{L^\infty([0,T] \times\R)}+2+P(T)^2\right).
	\end{align*}
Inserting here the uniform bounds on $\|\partial_t A\|_{L^\infty([0,T] \times\R)}$ and $P(T)$ from \eqref{eq:A-bound} and \eqref{eq:P-bound}, respectively, we conclude a uniform bound independent of $T$ on 
$\|\partial_x\partial_t A\|_{L^\infty([0,T] \times \R)}$, and the proof is complete.
\end{proof}

The bound for $\|\partial_x\partial_t A\|_{L^\infty([0,T] \times\R)}$ proved in  Proposition \ref{prop:bound-Atx} leads immediately to a uniform bound  for the derivatives of $f$, so that the condition \eqref{eq:lifespan-cond} is verified,
	\begin{equation*}
	(f,\partial_t A)\in W^{1,\infty}([0, T^*)\times \R^2)\times W^{1,\infty}([0, T^*)\times \R),
	\end{equation*}
and the solution is global.
This completes the proof of Theorem \ref{thm:global}.
\end{proof}

\appendix
\section{Division lemma} 

\begin{lemma}[Division lemma, originally in \cite{Bouchut2003}, appearing  in this particular form in \cite{Gerard2010}]\label{lem:division}
Let $Y(t,x)=\frac12 \ONE_{\{|x|\leq t\}}$ be the forward fundamental solution of the $1d$ wave operator and $a(v)\in(-1,1)$. Then the equality
\begin{equation}\label{eq:div1}
\partial_x^2Y=(\partial_t+a(v)\partial_x)\left(\frac{x}{a(v)x-t}\partial_x Y\right)+\frac1{a(v)^2-1}\delta_{(t,x)=(0,0)}
\end{equation}
holds in $\mathcal{D}'(\R^2)$.
\end{lemma}

\begin{proof}
Denoting $m(t,x)=\frac{x}{ax-t}$ and $T=\partial_t+a\partial_x$ the identity \eqref{eq:div1} which we seek to prove can be rewritten as
	\begin{equation}\label{eq:div2}
	T\left(m\partial_xY\right)
	=
	-\frac1{a(v)^2-1}\delta_{(t,x)=(0,0)}+\partial_x^2Y.
	\end{equation}
Recalling that $\partial_x Y(t,x)=\frac12\delta_{x=-t}-\frac12\delta_{x=t}$, one needs to be clear about the meaning of the left side of \eqref{eq:div2}. Having  the terms $\delta_{x=-t}$ and $\delta_{x=t}$, combined  with the restriction $|a|\neq1$, means that $m\partial_xY\in\mathcal{D}'(\R^2\setminus\{(0,0)\})$ is a well-defined distribution which is homogeneous of order $-1$. It admits a unique extension as a homogeneous distribution of order $-1$  in $\mathcal{D}'(\R^2)$ which we still denote $m\partial_xY$. Working with test functions, we observe that for every $\phi\in\mathcal{D}(\R^2)$
	\begin{equation*}
	\left<m\partial_xY,\phi\right>
	=
	\frac12\frac{1}{a+1}\int_0^\infty\phi(t,-t)\,dt
	-\frac12\frac{1}{a-1}\int_0^\infty\phi(t,t)\,dt.
	\end{equation*}
We can now compute $\left<T\left(m\partial_xY\right),\phi\right>$ by integrating by parts, and using the observation that
	\begin{equation*}
	T=\partial_t+a\partial_x=\partial_t+\partial_x+(a-1)\partial_x = \partial_t-\partial_x+(a+1)\partial_x.
	\end{equation*}
Specifically, we obtain
	\begin{align*}
	\left<T\left(m\partial_xY\right),\phi\right>
	&=	
	\frac12\frac{1}{a+1}\int_0^\infty(-T\phi)(t,-t)\,dt
	-\frac12\frac{1}{a-1}\int_0^\infty(-T\phi)(t,t)\,dt\\
	&=
	-\frac{1}{a^2-1}\phi(0,0)+\frac12\int_0^\infty\left(\partial_x\phi(t,t)-\partial_x\phi(t,-t)\right)dt.
	\end{align*}
The proof is complete by observing  that the  last term is precisely $\left<\partial_x^2Y,\phi\right>$. Indeed,
	\begin{equation*}
	\left<\partial_x^2Y,\phi\right>
	=
	\frac12\int_0^\infty\int_{-t}^t\partial_x^2\phi(t,x)\,dx\,dt
	=
	\frac12\int_0^\infty\left(\partial_x\phi(t,t)-\partial_x\phi(t,-t)\right)dt.
	\end{equation*}

\end{proof}

\bibliography{library}

\begin{thebibliography}{BGP03}

\bibitem[BGP03]{Bouchut2003}
Fran{\c{c}}ois Bouchut, Fran{\c{c}}ois Golse, and Christophe Pallard.
\newblock {Classical Solutions and the Glassey-Strauss Theorem for the 3D
  Vlasov-Maxwell System}.
\newblock {\em Archive for Rational Mechanics and Analysis}, 170(1):1--15, nov
  2003.

\bibitem[GP10]{Gerard2010}
Patrick G{\'{e}}rard and Christophe Pallard.
\newblock {A mean-field toy model for resonant transport}.
\newblock {\em Kinetic and Related Models}, 3(2):299--309, may 2010.

\bibitem[GS86]{Glassey1986}
Robert~T. Glassey and Walter~A. Strauss.
\newblock {Singularity formation in a collisionless plasma could occur only at
  high velocities}.
\newblock {\em Archive for Rational Mechanics and Analysis}, 92(1):59--90,
  1986.

\bibitem[GS87]{Glassey1987a}
Robert~T. Glassey and Walter~A. Strauss.
\newblock {Absence of Shocks in an Initially Dilute Collisionless Plasma}.
\newblock {\em Communications in mathematical physics}, 113(2):191--208, 1987.

\bibitem[GS88]{Glassey1988}
Robert~T. Glassey and Jack~W. Schaeffer.
\newblock {Global existence for the relativistic Vlasov-Maxwell system with
  nearly neutral initial data}.
\newblock {\em Communications In Mathematical Physics}, 119(3):353--384, sep
  1988.

\bibitem[GS90]{Glassey1990}
Robert~T. Glassey and Jack~W. Schaeffer.
\newblock {On the `one and one-half dimensional' relativistic Vlasov-Maxwell
  system}.
\newblock {\em Mathematical Methods in the Applied Sciences}, 13(2):169--179,
  aug 1990.

\bibitem[GS97]{Glassey1997}
Robert~T. Glassey and Jack~W. Schaeffer.
\newblock {The ``Two and One-Half Dimensional" Relativistic Vlasov Maxwell
  System}.
\newblock {\em Communications in mathematical physics}, 185(2):257--284, 1997.

\bibitem[KS02]{Klainerman2002}
Sergiu Klainerman and Gigliola Staffilani.
\newblock {A new approach to study the Vlasov-Maxwell system}.
\newblock {\em Communications on Pure and Applied Analysis}, 1(1):103--125,
  2002.

\bibitem[LS14]{Luk2014}
Jonathan Luk and Robert~M. Strain.
\newblock {A new continuation criterion for the relativistic Vlasov-Maxwell
  system}.
\newblock {\em Communications in mathematical physics}, 331(3):1005--1027, jun
  2014.

\bibitem[NP14]{Nguyen2014c}
Charles Nguyen and Stephen Pankavich.
\newblock {A one-dimensional kinetic model of plasma dynamics with a transport
  field}.
\newblock {\em Evolution Equations and Control Theory}, 3(4):681--698, oct
  2014.

\bibitem[Pfa92]{Pfaffelmoser1992}
K~Pfaffelmoser.
\newblock {Global classical solutions of the Vlasov-Poisson system in three
  dimensions for general initial data}.
\newblock {\em Journal of Differential Equations}, 95(2):281--303, feb 1992.

\bibitem[Sch86]{Schaeffer1986}
Jack~W. Schaeffer.
\newblock {The Classical Limit of the Relativistic Vlasov-Maxwell System}.
\newblock {\em Communications in mathematical physics}, 104(3):403--421, 1986.

\bibitem[Sch91]{Schaeffer1991}
Jack~W. Schaeffer.
\newblock {Global existence of smooth solutions to the vlasov poisson system in
  three dimensions}.
\newblock {\em Communications in Partial Differential Equations},
  16(8-9):1313--1335, jan 1991.

\bibitem[Wol84]{Wollman1984}
Stephen Wollman.
\newblock {An existence and uniqueness theorem for the Vlasov-Maxwell system}.
\newblock {\em Communications on Pure and Applied Mathematics}, 37(4):457--462,
  jul 1984.

\end{thebibliography}
\bibliographystyle{alpha}
\end{document}